\documentclass[11p,leqno]{amsart}
\usepackage{amsmath}
\usepackage{amsfonts}
\usepackage{amssymb}
\usepackage{graphicx}
\usepackage{color}
\usepackage{amsfonts,mathtools,mathrsfs}
\usepackage{amscd,esint}
\usepackage{amsmath,amsthm}
\usepackage{hyperref}
\usepackage{url}
\usepackage[backrefs]{amsrefs}
\usepackage{color}
\newtheorem{theorem}{Theorem}[section]\newtheorem{thm}[theorem]{Theorem}
\newtheorem*{theorem*}{Theorem}
\newtheorem{lemma}{Lemma}[section]
\newtheorem{corollary}[theorem]{Corollary}
\newtheorem{proposition}{Proposition}[section]

\def\Ric{\text{Ric}}

\def\l{\lambda}

\def\R{\mathbb R}

\def\vp{\varphi}

\def\Vol{\operatorname{Vol}}
\def\id{\operatorname{id}}
\def\Ric{\operatorname{Ric}}

\def\tr{\operatorname{tr}}

\numberwithin{equation}{section}

\begin{document}
\title[]{\bf Four-Dimensional Gradient Shrinking Solitons with Positive Isotropic Curvature}

\author{Xiaolong Li}
\address{Department of Mathematics, University of California, San Diego, La Jolla, CA 92093, USA}
\email{xil117@ucsd.edu}

\author{Lei Ni}
\address{Department of Mathematics, University of California, San Diego, La Jolla, CA 92093, USA}
\email{lni@math.ucsd.edu}

\author{Kui Wang}
\address{School of Mathematic Sciences, Soochow University, Suzhou, 215006, China}
\email{kuiwang@suda.edu.cn}

\thanks{ The research of the first author is partially supported by an Inamori fellowship and NSF grant DMS-1401500.  The research of the  second author is partially supported by NSF grant DMS-1401500.}
%\date{}

\maketitle

\begin{abstract}
We show that a four-dimensional complete gradient shrinking Ricci soliton with positive isotropic curvature
is either a quotient of $\mathbb{S}^4$ or a quotient of $\mathbb{S}^3 \times \R $. This gives a clean classification result   removing the earlier additional assumptions in \cite{NW08} by  Wallach and the second author.
\end{abstract}

\section{Introduction}
A gradient shrinking Ricci soliton is a triple $(M,g,f)$, a complete Riemannian manifold with a smooth potential function $f$ whose Hessian satisfying
\begin{equation}\label{eq:soliton1}
  \Ric +\nabla \nabla f =\frac 1 2 g.
\end{equation}
 The gradient shrinking soliton arises naturally in the study of the singularity analysis of the Ricci flow \cite{hamilton-sin, Perelman}. It also attracts the study since (\ref{eq:soliton1}) generalizes the Einstein metrics.  The main purpose of this article concerns a classification of such four-manifolds with positive isotropic curvature.
The positive isotropic curvature condition was first introduced by Micalleff and Moore \cite{MM88} in applying the index computation of  harmonic spheres to the study of the topology of manifolds. The Ricci flow on four-manifolds with positive isotropic curvature was studied by Hamilton \cite{Ham97}. This condition was proven to be invariant under Ricci flow in dimension four by Hamilton and in high dimensions by Brendle-Schoen \cite{BS09} and Nguyen \cite{Ng10}. It is hence then interesting to understand the solitons under the positive isotropic curvature condition. In \cite{NW08}, Wallach and the second author classified the four-dimensional gradient shrinking solitons under some extra assumptions, including the nonnegative curvature operator, a pinching condition and a curvature growth condition. Since then, there have been much progresses in understanding the general shrinking solitons \cite{MS13, MW15b, KW, Naber, CZ10} and particularly the four-dimensional ones \cite{MW15a}. In particular, a classification result was obtained in \cite{MW15b} for solitons with nonnegative curvature operator for all dimensions. The purpose of this article is to prove the following classification result on shrinking solitons with positive isotropic curvature by removing all the additional assumptions in \cite{NW08}.

\begin{thm}\label{Main Thm}
Any four dimensional complete gradient shrinking Ricci soliton with positive isotropic curvature
is either a quotient of $\mathbb{S}^4$ or a quotient of $\mathbb{S}^3 \times \R $.
\end{thm}

It remains an interesting question whether or not the same result holds in high dimensions. We plan to return to this in the future.

\section{Preliminaries}
In this section, we collect some results on gradient shrinking Ricci solitons that will be used in this paper.

After normalizing the potential function $f$ via translating, we have the following identities \cite{hamilton-sin}:
\begin{lemma}\label{basic equ}
\begin{equation*}
  S+\Delta f =\frac{n}{2},
\end{equation*}
\begin{equation*}
  S+|\nabla f|^2 =f,
\end{equation*}
where $S$ denotes the scalar curvature of $M$.
\end{lemma}

Regarding the growth of the potential function $f$ and the volume of geodesic balls, Cao and Zhou \cite{CZ10} showed that
\begin{lemma}\label{potential func}
Let $(M^n, g)$ be a complete gradient shrinking Ricci soliton of dimension $n$ and $p \in M$. Then there are positive constants $c_1, c_2$ and $C$ such that
\begin{equation*}
  \frac{1}{4} \left(d(x,p)-c_1\right)^2_{+} \leq f(x) \leq  \frac{1}{4} \left(d(x,p)+c_2\right)^2,
\end{equation*}
\begin{equation*}
  \Vol(B_p(r)) \leq Cr^n.
\end{equation*}
\end{lemma}

Munteanu and Sesum \cite{MS13} proved the following integral bound for the Ricci curvature.
\begin{lemma}\label{Ric int bound}
Let $(M, g)$ be a complete gradient shrinking Ricci soliton. Then for any $\l >0$, we have
$$\int_{M} |\Ric|^2 e^{-\l f} < \infty.$$
\end{lemma}

The above mentioned results are valid in all dimensions. In the following, we recall some special properties in four dimension.
It is well known that, in dimension four, the curvature operator $R$ can be written as
$$R =\begin{pmatrix}
A & B \\
B^t & C
\end{pmatrix}$$
according to the natural splitting $\wedge^2(\R^4) =\wedge_{+} \oplus \wedge_{-}$,
where $\wedge_{+}$ and $\wedge_{-}$ are the self-dual and anti-self-dual parts respectively.
It is easy to see that $A$ and $C$ are symmetric.
Denote $A_1\leq A_2 \leq A_3$ and $C_1\leq C_2 \leq C_3$ the eigenvalues of $A$ and $C$, respectively.
Also let $B_1 \leq B_2 \leq B_3$ be the singular eigenvalues of $B$.
Note that a direct consequence of the first Bianchi identity is that $\tr(A) =\tr(C) =\frac{S}{4}$, where $S$ is the scalar curvature.
In \cite{NW08}, Wallach and the second author computed that $\overset{\circ}{\Ric}$ can be expressed in terms of components of $B$,
where $\overset{\circ}{\Ric}$ is the traceless part of the Ricci tensor.
In particular, $4\|B\|^2 =|\overset{\circ}{\Ric}|^2$ and $\sum_{1}^{4} \l_i^3 =24 \det B$, where $\l_i$'s are the eigenvalues of $\overset{\circ}{\Ric}$.
It was also observed in \cite[Proposition 3.1]{NW08} that these components of the curvature operator satisfy certain differential inequalities,
which play a significant role in the classification results.
\begin{proposition}\label{diff inequ}
If $(M, g(t))$ is a solution to the Ricci flow, then we have the following differential inequalities
\begin{eqnarray*}
  \left(\frac{\partial}{\partial t} - \Delta\right) (A_1+A_2) &\geq & A_1^2 +A_2^2 +2(A_1+A_2)A_3 +B_1^2 +B_2^2, \\
  \left(\frac{\partial}{\partial t} - \Delta\right) (C_1+C_2) &\geq & C_1^2 +C_2^2 +2(C_1+C_2)C_3 +B_1^2 +B_2^2,\\
  \left(\frac{\partial}{\partial t} - \Delta\right) B_3 &\leq & A_3B_3 +C_3B_3 +2B_1B_2
\end{eqnarray*}
in the distributional sense.% , where $A_i, B_i$ and $C_i$ are components of the curvature operator.
\end{proposition}

We now summarize the ideas  used to prove the classification result in \cite{NW08}
and explain our strategy of removing all the assumptions except positive isotropic curvature.
First of all, it was shown in \cite[Proposition 4.2]{NW08a} that if $S\neq 0$, then
\begin{equation*}
  \left(\frac{\partial }{\partial t} -\Delta \right) \left(\frac{|R_{ijkl}|^2}{S^2} \right)= \frac{4P}{S^3} -\frac{2}{S^4}|S\nabla_pR_{ijkl}-\nabla_pS R_{ijkl}|^2
  +\left\langle \nabla \left(\frac{|R_{ijkl}|^2}{S^2} \right), \nabla \log{S^2} \right\rangle,
\end{equation*}
where $P$ is defined by
\begin{equation*}
  P=4 S \langle R^2 +R^{\sharp}, R\rangle  -|\Ric|^2 |R_{ijkl}|^2.
\end{equation*}
Here $\#$ is the operator defined using the Lie algebraic structure of $\wedge^2 T_pM$ (see for example \cite{hamilton-sin}).
Under a certain curvature growth assumption, the classification result follows from the proof of the main theorem in \cite{NW08a} if one can show that $P \leq 0 $.
In dimension four, $P$ can be expressed in terms of $A, B$ and $C$.
Let $\overset{\circ}{A}$ and $\overset{\circ}{C}$ be the traceless parts of $A$ and $C$, respectively.
By choosing suitable basis of $\wedge_{+}$ and $\wedge_{-}$, we may diagonalize $\overset{\circ}{A}$  and $\overset{\circ}{C}$ such that
\begin{equation*}
  A=\begin{pmatrix}
    \frac{S}{12}+a_1 & 0 & 0 \\
    0 & \frac{S}{12}+a_2 & 0 \\
    0 & 0 & \frac{S}{12}+a_3
  \end{pmatrix}, \ \
  C=\begin{pmatrix}
    \frac{S}{12}+c_1 & 0 & 0 \\
    0 & \frac{S}{12}+c_2 & 0 \\
    0 & 0 & \frac{S}{12}+c_3
  \end{pmatrix}.
\end{equation*}
 Then $P$ can be written as (See \cite{NW08})
\begin{eqnarray*}
  P &=& -S^2\left(\frac 1 6 \sum_{1}^{4}\l_i^2 +\sum_{1}^{3}a_i^2 +\sum_{1}^{3} c_i^2 \right)
  \\&\,&+4S \left(\sum_{1}^{3}(a_i^3+c_i^3) +6a_1a_2a_3 +6c_1c_2c_3 -\frac 1 2 \sum_{1}^{4} \l_i^3\right)  \\
   &\,& +12S\left(a_1b_1^2 +a_2b_2^2 +a_3b_3^2 +c_1 \tilde{b}_1^2  +c_2 \tilde{b}_2^2+c_3 \tilde{b}_3^2 \right) \\
   &\,& - 2\left(\sum_{1}^{4}\l_i^2\right)^2 -4 \left(\sum_{1}^{4}\l_i^2 \right) \left(\sum_{1}^{3}(a_i^2+c_i^2)\right) ,
\end{eqnarray*}
where $\sum_{1}^{3}a_i =\sum_{1}^{3}c_i =\sum_{1}^{4}\l_i =0$, $b_i^2 =\sum_{j=1}^{3}B_{ij}^2$ and $\tilde{b}_i^2 =\sum_{j=1}^{3}B_{ji}^2$.

A key observation in \cite{NW08} is that for several basic examples, one always has $BB^t =b^2 \id $ and $P=0$.
By considering the evolution equation of the quantity $\frac{B_3^4}{(A_1+A_2)^2(C_1+C_2)^2}$ and applying integration by parts,
Wallach and the second author proved that $BB^t =b^2 \id $ is indeed valid for all four-dimensional shrinking Ricci solitons satisfying two very weak curvature assumptions:
\begin{equation}\label{eq:ass1}
  \frac{B_3^2}{(A_1+A_2)(C_1+C_2)} (x) \leq \exp(a (r(x) +1)),
\end{equation}
\begin{equation}\label{eq:ass2}
 | R_{ijkl}|(x)\leq \exp(a (r(x) +1))
\end{equation}
for some $a>0$, where $r(x)$ is the distance function to a fixed point on the manifold.
These two assumptions are only needed to ensure that all the integrals involved in the integration by parts argument are finite.
To remove these two curvature assumptions, we show that $BB^t =b^2 \id $ in fact holds on all four-dimensional shrinking Ricci solitons.
Our strategy is to consider the evolution equation of $\frac{B_3}{\sqrt{(A_1+A_2)(C_1+C_2)}}$ and sharpen the integration by parts argument in \cite{NW08}. 
In the proof, in order to guarantee that all the integrals involved are finite, we use the  integral bound of the Ricci curvature in Lemma \ref{Ric int bound}.

Now under the assumption $BB^t =b^2 \id$, $P$ has a much simpler expression:
\begin{eqnarray*}
  P &=& -S^2\left(\sum_{1}^{3}a_i^2 +\sum_{1}^{3} c_i^2 \right) +12S\left(\sum_{1}^{3}(a_i^3+c_i^3)\right) \\
     && -2b^2\left(S+12b\right)^2 - 48b^2 \left(\sum_{1}^{3}a_i^2 +\sum_{1}^{3} c_i^2 \right)  \\
   &\leq & -S\left(S\sum_{1}^{3}a_i^2- 12\sum_{1}^{3}a_i^3 \right) -S\left(S\sum_{1}^{3}c_i^2- 12\sum_{1}^{3}c_i^3 \right),
\end{eqnarray*}
where we have used the following
\begin{eqnarray*}
&& 3b^2 =\sum_{1}^{3}b_i^2 =\sum_{1}^{3}\tilde{b}_i^2 =\frac 1 4\sum_{1}^{4}\l_i^2,\\
&& \sum_{1}^{3}a_i^3 +6a_1a_2a_3 =3 \sum_{1}^{3}a_i^3,\\
&& \sum_{1}^{4} \l_i^3 =24 \det B =24 b^3.
\end{eqnarray*}
Then by further assuming that $M$ has nonnegative curvature operator, namely, $A$ and $C$ are positive semidefinite, $P\leq 0$ was proved in \cite{NW08} by solving two optimization problems with constraints.
More precisely, they showed that under the constraints $\sum_{1}^{3}a_i = 0$ and $\frac{S}{12} +a_i \geq 0$,
the maximum of $\sum_{1}^{3} a_i^2$ is $\frac{S^2}{24}$, while under further normalization $\sum_{1}^{3}a_i^2 =1$ it holds that
\begin{equation*}
  \frac{\sum_{1}^{3}a_i^3}{\sum_{1}^{3}a_i^2} \leq \frac{1}{\sqrt{6}} \left(\sum_{1}^{3}a_i^2\right)^{\frac 1 2}.
\end{equation*}
Combining these results, one easily obtains
\begin{equation*}
  S\sum_{1}^{3}a_i^2- 12\sum_{1}^{3}a_i^3 \geq 0.
\end{equation*}
The terms with $c_i$' s can be handled similarly and thus they arrived at that $P\leq 0$.

By solving another optimization problem under weaker constraints,
we achieve in this paper that $P\leq 0$ without assuming $M$ has nonnegative curvature operator. In the last step  we also sharpen the integration by parts argument in \cite{NW08a} instead of  appealing the result in \cite{NW08a} as in \cite{NW08} since we no longer assume a curvature growth condition (\ref{eq:ass2}). These three improvements allow us to prove the main result without any additional assumptions.

\section{Proof of Theorem \ref{Main Thm} }
For brevity, we introduce the same notations as in \cite{NW08}: $\psi_1=A_1+A_2$, $\psi_2=C_1+C_2$, $\vp=B_3$ and
\begin{eqnarray*}
  -E &=& -\frac{4B_1(B_3-B_2)}{B_3}-\frac{(A_1-B_1)^2+(A_2-B_2)^2+2A_2(B_2-B_1)}{A_1+A_2} \\
   && -\frac{(C_1-B_1)^2+(C_2-B_2)^2+2C_2(B_2-B_1)}{C_1+C_2}.
\end{eqnarray*}
It is clear that $-E \leq 0$ with equality holds only if $A_1 =C_1 =B_1 =B_2 =A_2 =C_2 =B_3$.
In particular we have $BB^t =b^2 \id$ for some $b$.
Also notice that $M$ has positive isotropic curvature amounts to $\psi_1 >0$ and $\psi_2 >0$.

%Let $v=\frac{\vp}{\sqrt{\psi_1 \psi_2}}. $
\begin{proposition}
The following differential inequality holds in the sense of distribution:
\begin{eqnarray}\label{main diff inequ}
\left(\frac{\partial}{\partial t} - \Delta \right)\frac{\vp}{\sqrt{\psi_1 \psi_2}}
&\leq&  -\frac{1}{2}\frac{\vp}{\sqrt{\psi_1 \psi_2}} E
-\frac{1}{4} \frac{\vp |\psi_1 \nabla \psi_2 -\psi_2 \nabla \psi_1 |^2}{{(\psi_1 \psi_2)}^{\frac{5}{2}}} \\
&& +\left\langle \nabla \left( \frac{\vp}{\sqrt{\psi_1 \psi_2}}\right), \nabla \log(\psi_1 \psi_2) \right\rangle. \nonumber
\end{eqnarray}
\end{proposition}

\begin{proof}
Straightforward calculations yield
\begin{eqnarray*}
  \left(\frac{\partial}{\partial t} - \Delta \right)\frac{\vp}{\sqrt{\psi_1 \psi_2}}
   &=&  \frac{\left(\frac{\partial}{\partial t} - \Delta \right)\vp}{\sqrt{\psi_1 \psi_2}}
   -\frac{1}{2} \frac{\vp \psi_1 \left(\frac{\partial}{\partial t} - \Delta \right)\psi_2
   +\vp \psi_2 \left(\frac{\partial}{\partial t} - \Delta \right)\psi_1 }{{(\psi_1 \psi_2)}^{\frac{3}{2}} }  \\
   && %+ \frac{\vp \langle \nabla \psi_1, \nabla \psi_2\rangle }{{(\psi_1 \psi_2)}^{\frac{3}{2}}}
   -\frac{1}{4} \frac{\vp |\psi_1 \nabla \psi_2 -\psi_2 \nabla \psi_1 |^2}{{(\psi_1 \psi_2)}^{\frac{5}{2}}}
   %&&
   +\left\langle \nabla \left(\frac{\vp}{\sqrt{\psi_1 \psi_2}}\right), \nabla \log(\psi_1 \psi_2) \right\rangle
\end{eqnarray*}

Substituting the differential inequalities in Proposition \ref{diff inequ} into the above equation gives, after some cancelations, that
\begin{eqnarray*}
  \left(\frac{\partial}{\partial t} - \Delta \right)\frac{\vp}{\sqrt{\psi_1 \psi_2}}
   &\leq&  -\frac{1}{2}\frac{\vp}{\sqrt{\psi_1 \psi_2}} E
    -\frac{1}{4} \frac{\vp |\psi_1 \nabla \psi_2 -\psi_2 \nabla \psi_1 |^2}{{(\psi_1 \psi_2)}^{\frac{5}{2}}}\\
   && +\left\langle \nabla \left(\frac{\vp}{\sqrt{\psi_1 \psi_2}}\right), \nabla \log(\psi_1 \psi_2) \right\rangle.
\end{eqnarray*}
This proves the proposition.
\end{proof}

\begin{proposition}\label{BB^t}
Let $M$ be a four-dimensional gradient shrinking Ricci soliton with positive isotropic curvature. Then $BB^t =b^2 \id$.
\end{proposition}
\begin{proof}
  Note that on a gradient Ricci soliton, it holds that
  \begin{equation*}
    \frac{\partial }{\partial t} \left(\frac{\vp}{\sqrt{\psi_1 \psi_2}} \right)=\left\langle \nabla f, \nabla \left(\frac{\vp}{\sqrt{\psi_1 \psi_2}}\right) \right\rangle.
  \end{equation*}
  Fix a point $x_0$ in $M$. For any $r >0$, one can choose a smooth cut-off function $\eta$ with support in $\{x\in M : d(x, x_0)\leq r\}$ and $|\nabla \eta | \leq C/r$.
  Then multiplying both sides of \eqref{main diff inequ} by $ \frac{\vp}{\sqrt{\psi_1 \psi_2}} e^{-f +\log (\psi_1 \psi_2)} \eta^2 $ and integrating over $M$,
  \begin{eqnarray*}
    && \int_{M}\left(\left\langle \nabla f, \nabla \left(\frac{\vp}{\sqrt{\psi_1 \psi_2}}\right) \right\rangle - \Delta \left( \frac{\vp}{\sqrt{\psi_1 \psi_2}} \right) \right)\frac{\vp}{\sqrt{\psi_1 \psi_2}}
   % \frac{\vp}{\sqrt{\psi_1 \psi_2}}  e^{-f +\log (\psi_1 \psi_2) }\eta^2
    %-\int_{M} \Delta \left( \frac{\vp}{\sqrt{\psi_1 \psi_2}} \right) \frac{\vp}{\sqrt{\psi_1 \psi_2}}
    e^{-f +\log (\psi_1 \psi_2) }\eta^2 \\
     &\leq & -\frac{1}{2} \int_{M} \vp^2  E e^{-f} \eta^2
     + \int_{M} \left\langle \nabla \left( \frac{\vp}{\sqrt{\psi_1 \psi_2}}\right), \nabla \log(\psi_1 \psi_2) \right\rangle \frac{\vp}{\sqrt{\psi_1 \psi_2}}  e^{-f +\log (\psi_1 \psi_2)} \eta^2
  \end{eqnarray*}

After integration by parts and some cancelations, we arrive at
\begin{eqnarray*}
   0 &\leq & -\int_{M}  \left|\nabla \left( \frac{\vp}{\sqrt{\psi_1 \psi_2}} \right) \right|^2 e^{-f +\log (\psi_1 \psi_2)} \eta^2
   -\frac{1}{2} \int_{M} \vp^2 E e^{-f} \eta^2 \\
   &&-2\int_{M}\left\langle \nabla \left(\frac{\vp}{\sqrt{\psi_1 \psi_2}}\right), \nabla \eta \right\rangle \frac{\vp}{\sqrt{\psi_1 \psi_2}}  e^{-f +\log (\psi_1 \psi_2) } \eta \\
   &\leq &  \int_{M} \vp^2 |\nabla \eta|^2  e^{-f}  -\frac{1}{2} \int_{M} \vp^2 E e^{-f} \eta^2 \\
   &\leq & \frac{C^2}{r^2} \int_{M} \vp^2 e^{-f} -\frac{1}{2} \int_{M} \vp^2 E e^{-f} \eta^2.
\end{eqnarray*}
Since $4 \vp^2 = 4 B_3^2 \leq 4 \|B\|^2 = |\overset{\circ}{\Ric}|^2=|\Ric|^2-\frac{S^2}{4} $,
we obtain, in view of Lemma \ref{basic equ}, Lemma \ref{potential func} and Lemma \ref{Ric int bound},
\begin{equation*}
  \int_{M} \vp^2 e^{-f} \leq \frac{1}{4} \int_{M} \left(|\Ric|^2-\frac{S^2}{4}\right)e^{-f} <\infty.
\end{equation*}
Letting $r \to \infty$ implies that $\int_{M} \vp^2 E e^{-f}=0$. Therefore, we must have either $B_3=0$ or $E=0$. It then follows that $BB^t =b^2 \id$.
\end{proof}

\begin{proposition}
Let $M$ be a four-dimensional gradient shrinking Ricci soliton with positive isotropic curvature. Then $P\leq 0.$
\end{proposition}
\begin{proof}
  Recall from Section 2, since $BB^t = b^2 \id $, we have
  \begin{eqnarray*}
  P \leq  -S\left(S\sum_{1}^{3}a_i^2- 12\sum_{1}^{3}a_i^3 \right) -S\left(S\sum_{1}^{3}c_i^2- 12\sum_{1}^{3}c_i^3 \right).
\end{eqnarray*}
In order to prove $P\leq 0$, it suffices to show that
\begin{equation*}
S\sum_{1}^{3}a_i^2- 12\sum_{1}^{3}a_i^3 \geq 0.
\end{equation*}
With suitable choices of the orthonormal basis for $\wedge_+$, we can assume that $A_i =\frac{S}{12} +a_i$.% and $C_i =\frac{S}{12} +c_i$ and
Note that we have the constraints $\sum_{1}^{3}a_i =0$ and $A_i+A_j =\frac{S}{6}+a_i+a_j > 0$  for $i\neq j$ because $M$ has positive isotropic curvature.
By the change of variables $x_i =1- \frac{6}{S}a_i$, the constraints become $\sum_{1}^{3} x_i =3$ and $x_i > 0$ for $1\leq i \leq 3$,
and the objective function becomes
\begin{eqnarray*}
F(x_1,x_2,x_3) &:= &S\sum_{1}^{3}a_i^2- 12\sum_{1}^{3}a_i^3 =  \frac{S^3}{36}\left(\sum_{1}^{3} (1-x_i)^2 -2(1-x_i)^3\right)  \\
  &=& \frac{S^3}{36} \left(2 \sum_{1}^{3} x_i^3- 5 \sum_{1}^{3} x_i^2 +9\right).
\end{eqnarray*}

Using Lagrange multipliers, we find two critical points $Z=(1,1,1)$ and $W=\left(\frac{1}{3}, \frac{4}{3},  \frac{4}{3}\right) $
with $F(Z)= 0$ and $F(W)=\frac{S^3}{162} $. On the boundary, we have $x_i=0$.
Since $F$ is symmetric, we can assume without loss of generality that $x_1 =0$.
Then we have, using $x_3 =3-x_2$, that
$$F(x_1,x_2,x_3) = \frac{S^3}{36} \left(2(x_2^3+(3-x_2)^3 ) -5(x_2^2+(3-x_2)^2 ) +9\right) = \frac{S^3}{18}(2x_2 -3)^2 \geq 0.$$
Therefore, under the constraints  $\sum_{1}^{3}a_i =0$ and $\frac{S}{6}+a_i+a_j \geq 0$ for $i\neq j$,
$$F(x_1,x_2,x_3) = S\sum_{1}^{3}a_i^2- 12\sum_{1}^{3}a_i^3 \geq 0.$$
The terms involving $c_i$'s can be handled similarly. Hence $P\leq 0$.
\end{proof}

\begin{proof}[Proof of Theorem \ref{Main Thm}]
Recall it was shown in \cite[Proposition 4.2]{NW08a} that if $S\neq 0$, then
\begin{equation}\label{evol}
  \left(\frac{\partial }{\partial t} -\Delta \right) \left(\frac{|R_{ijkl}|^2}{S^2} \right)= \frac{4P}{S^3} -\frac{2}{S^4}|S\nabla_pR_{ijkl}-\nabla_pS R_{ijkl}|^2
  +\left\langle \nabla \left(\frac{|R_{ijkl}|^2}{S^2} \right), \nabla \log{S^2} \right\rangle,
\end{equation}
Trying to use integration by parts here would require a stronger integral bound of the Ricci curvature than we actually have in Lemma \ref{Ric int bound}.
To overcome this difficulty, we adopt a similar idea that was used to prove $BB^t =b^2 \id$.
We consider, instead, $u =\frac{|R_{ijkl}|}{S}$ and $T=\frac{R_{ijkl}}{S}$.
A direct calculation shows that
\begin{eqnarray*}
  \left(\frac{\partial }{\partial t} -\Delta \right) u
  &=& \frac{2P}{uS^3} +\langle \nabla u, \nabla \log{S^2} \rangle + \frac{|\nabla u|^2-|\nabla T|^2}{u}\\
   &\leq &  \frac{2P}{uS^3} +\langle \nabla u, \nabla \log{S^2} \rangle\\
\end{eqnarray*}
where we have used Kato's inequality in the last line.
Now let $\eta$ be a smooth cut-off function with support in $\{x\in M : d(x, x_0)\leq r\}$ and $|\nabla \eta | \leq C/r$.
Multiplying the above inequality by $u e^{-f+\log{S^2} }\eta^2 $ and integrating over $M$, we obtain,
\begin{eqnarray*}
&& \int_{M} \langle \nabla f, \nabla u \rangle u e^{-f+\log{S^2} }\eta^2 -\int_{M} \Delta u \mbox{ } u e^{-f+\log{S^2} }\eta^2  \\
 &\leq & \int_{M} \frac{2P}{S^3}  e^{-f+\log{S^2} }\eta^2  + \int_{M}  \langle \nabla u, \nabla \log{S^2} \rangle u e^{-f+\log{S^2} }\eta^2.
\end{eqnarray*}
Integration by parts then yields
\begin{eqnarray*}
&&\frac{1}{2} \int_{M} |\nabla u|^2 e^{-f+\log{S^2} }\eta^2 -2\int_{M} \frac{P}{S} e^{-f} \\
   &\leq &  -\frac{1}{2} \int_{M} |\nabla u|^2 e^{-f+\log{S^2} }\eta^2 -2 \int_{M} \langle \nabla u, \nabla \phi \rangle u \eta e^{-f+\log{S^2} }\\
   &\leq & 8 \int_{M} |\nabla \eta |^2 u^2 e^{-f+\log{S^2}} \\
   &\leq & \frac{C}{r^2}  \int_{M}  |R_{ijkl}|^2 e^{-f}.
\end{eqnarray*}
The integral $\int_{M}  |R_{ijkl}|^2 e^{-f}$ is finite in view of Lemma \ref{Ric int bound},
since if $M$ has positive isotropic curvature, then the components of curvature operator $A, B$ and $C$ can be estimated by
\begin{eqnarray*}
 && -\frac{S}{4} \leq A_1 \leq A_2 \leq A_3 \leq \frac{S}{4}, \\
 && -\frac{S}{4} \leq C_1 \leq C_2 \leq C_3 \leq \frac{S}{4}, \\
 && 4 \|B\| \leq |\overset{\circ}{\Ric}|.
\end{eqnarray*}
Therefore we know that $u$ is a positive constant and $P=0$ by letting $r\to \infty$.
Then it follows from \eqref{evol} that $|S\nabla_pR_{ijkl}-\nabla_pS R_{ijkl}|^2=0$.
Theorem \ref{Main Thm} then follows from the proof of the main theorem in \cite{NW08a}.
\end{proof}

The strong maximum principle together with the classification of positive case implies the following corollary for the solitons with nonnegative isotropic curvature.

\begin{corollary} If $(M, g, f)$ is a complete gradient shrinking soliton with nonnegative isotropic curvature then its universal cover must be one of the following spaces $\mathbb{S}^4$, $\mathbb{C}P^2$, $\mathbb{S}^2\times \mathbb{S}^2$, $\mathbb{S}^2\times \mathbb{R}^2$ and $\mathbb{S}^3\times \mathbb{R}$.
\end{corollary}


\begin{bibdiv}
\begin{biblist}

\bib{BS09}{article}{
 author={Brendle, Simon},
 author={Schoen, Richard},
  title={Manifolds with 1/4-pinched curvature are space forms},
   journal={J. Amer. Math. Soc.},
   volume={ 22},
   date={2009},
   number={1},
   pages={287--307},
 review={\MR{2449060}}
}

\bib{CZ10}{article}{
    AUTHOR = {Cao, Huai-Dong},
    author={ Zhou, Detang},
     TITLE = {On complete gradient shrinking {R}icci solitons},
   JOURNAL = {J. Differential Geom.},
    VOLUME = {85},
      date = {2010},
    NUMBER = {2},
     PAGES = {175--185},
      ISSN = {0022-040X},
     review = {\MR{2732975}},
}

\bib{hamilton-sin}{article}
{
author={Hamilton, Richard S.},
 title={The formation of singularities in the Ricci flow},
  journal={Surveys in differential geometry},
   volume={ II},
   date={ 1993},
   pages={ 7--136},
   review={\MR{1375255}}
    }

\bib{Ham97}{article}{
author={Hamilton, Richard S.},
 title={Four-manifolds with positive isotropic curvature},
 journal= {Comm. Anal. Geom.},
 volume={ 5},
 date= {1997},
 number={  1},
 pages={ 1--92},
 review={\MR{1456308}}
}

\bib{KW}{article}{
 author={ Kotschwar, Brett},
 author={Wang, Lu},
  title={Rigidity of asymptotically conical shrinking gradient Ricci solitons},
  journal={J. Differential Geom.},
  volume={ 100},
  year={2015},
  number={ 1},
  pages= {55--108},
  review={\MR{3326574}}

}

\bib{MM88}{article}{
author={Micallef, Mario J.},
author={Moore, John Douglas},
title={Minimal two-spheres and the topology of manifolds with positive curvature on totally isotropic two-planes},
journal={ Ann. of Math. (2)},
volume= {127},
year={1988},
number={ 1},
pages={199--227},
review={\MR{0924677}}
}



\bib{MS13}{article}{
    AUTHOR = {Munteanu, Ovidiu},
author= {Sesum, Natasa},
     TITLE = {On gradient {R}icci solitons},
   JOURNAL = {J. Geom. Anal.},
    VOLUME = {23},
      date = {2013},
    NUMBER = {2},
     PAGES = {539--561},
      ISSN = {1050-6926},
review = {\MR{3023848}},
      DOI = {10.1007/s12220-011-9252-6},
}

\bib{MW15a}{article}{
author={Munteanu, Ovidiu},
author={Wang, Jiaping},
title={ Geometry of shrinking Ricci solitons},
journal={ Compos. Math.},
volume={ 151},
date={2015},
pages={ 2273--2300.}
}


\bib{MW15b}{article}{
author={Munteanu, Ovidiu},
author={Wang, Jiaping},
title={Positively curved shrinking Ricci solitons are compact},
status={ArXiv preprint.}
}

\bib{Naber}{article}{
author={Naber, Aaron},
 title={Noncompact shrinking four solitons with nonnegative curvature},
  journal={J. Reine Angew. Math.},
  volume={645},
  date={2010},
  pages={ 125--153},
  review={\MR{2673425}}
  }


\bib{Ng10}{article}
{
author={Nguyen, Huy T.},
title={
Isotropic curvature and the Ricci flow},
journal={
Int. Math. Res. Not. IMRN},
date={ 2010},
number={3},
pages={536–-558},
review={\MR{MR2587576}}

}


\bib{NW08a}{article}{
author = {Ni, Lei},
author={Wallach, Nolan},
     title = {On a classification of gradient shrinking solitons},
   journal = {Math. Res. Lett.},
   volume={15},
   date={2008},
   number={5},
   pages={941--955},
   issn={1073-2780},
   review={\MR{2443993}},
   doi={10.4310/MRL.2008.v15.n5.a9},
      }

\bib{NW08}{article}
{
author = {Ni, Lei},
author={Wallach, Nolan},
     title = {On four-dimensional gradient shrinking solitons},
   journal = {Int. Math. Res. Not. IMRN},
      date = {2008},
    number = {4},
     pages = {Art. ID rnm152, 13},
      issn = {1073-7928},
  review = {\MR{2424175 (2010g:53123)}},
       DOI = {10.1093/imrn/rnm152},
       URL = {http://dx.doi.org/10.1093/imrn/rnm152},
}
	
\bib{Perelman}{article}
{
 author={Perelman, Grisha},
 title={ Ricci flow with surgery on three-manifolds},
 status={arXiv:\ math.DG/\ 0303109}
}












\end{biblist}
\end{bibdiv}
\end{document}